\documentclass[11pt]{article} 
\usepackage{amsmath,amssymb,amsthm,color}
\newcommand{\EE}{\mathsf{E}}

\newcommand{\PP}{\mathsf{P}}
\newcommand{\NN}{\mathbb{N}}
\newcommand{\RR}{\mathbb{R}}
\newcommand{\ZZ}{\mathbb{Z}}

\newcommand{\Be}{\mathrm{Be}}

\newcommand{\Po}{\mathrm{Po}}
\newcommand{\ee}{\mathrm{e}}

\newcommand{\bA}{\boldsymbol{A}}

\newcommand{\bj}{\boldsymbol{j}}
\newcommand{\bk}{\boldsymbol{k}}
\newcommand{\bell}{\boldsymbol{\ell}}
\newcommand{\bmm}{\boldsymbol{m}}
\newcommand{\bp}{\boldsymbol{p}}

\newcommand{\be}{\boldsymbol{e}}
\newcommand{\bB}{\boldsymbol{B}}
\newcommand{\bC}{\boldsymbol{C}}

\newcommand{\bx}{\boldsymbol{x}}
\newcommand{\by}{\boldsymbol{y}}
\newcommand{\bX}{\boldsymbol{X}}
\newcommand{\bY}{\boldsymbol{Y}}

\newcommand{\bU}{\boldsymbol{U}}
\newcommand{\bI}{\boldsymbol{I}}
\newcommand{\bG}{\boldsymbol{G}}
\newcommand{\bH}{\boldsymbol{H}}

\newcommand{\cL}{\mathcal{L}}
\newcommand{\tc}{\widetilde{c}}

\newcommand{\distr}{\stackrel{\mathcal{D}}{\longrightarrow}}

\newcommand{\vare}{\varepsilon}
\newcommand{\varr}{\varrho}

\newcommand{\blambda}{\boldsymbol{\lambda}}
\newcommand{\bLambda}{\boldsymbol{\Lambda}}
\newcommand{\bvare}{\boldsymbol{\vare}}

\newcommand{\bxi}{\boldsymbol{\xi}}
\newcommand{\bTheta}{\boldsymbol{\Theta}}
\newcommand{\bzero}{\boldsymbol{0}}
\newcommand{\bbone}{\boldsymbol{1}}

\newcommand{\diag}{\operatorname{diag}}

\numberwithin{equation}{section}

\newtheorem{lemma}{Lemma}
\newtheorem{theorem}{Theorem}
\newtheorem{proposition}{Proposition}
\newtheorem{corollary}{Corollary}

\theoremstyle{definition}
\newtheorem{remark}{Remark}
\newtheorem{example}{Example}

\renewenvironment{proof}{\smallskip \noindent \textbf{Proof.}}
{\vspace*{1pt} \hfill $\square$ \medskip}

\newenvironment{proof*}[1]{\smallskip \noindent
\textbf{Proof of #1.}}
{\vspace*{1pt} \hfill $\square$ \medskip}

\begin{document}

\begin{center}
{\LARGE \textbf{
Asymptotic behavior of a multi-type nearly
critical Galton--Watson processes with immigration
}}
\end{center}

\bigskip

\begin{center}
\textsc{L\'aszl\'o Gy\"orfi$^1$, M\'arton Isp\'any$^2$,
P\'eter Kevei$^3$} and \textsc{Gyula Pap$^4$}
\end{center}

\smallskip

\begin{center}
$^1$Department of Computer Science and Information Theory,
Budapest University \\
of Technology and Economics, 
Stoczek u. 2, Budapest, Hungary, H-1521; \\
\textit{e-mail}: \texttt{gyorfi@szit.bme.hu}, \\
$^2$Department of Information Technology, 
Faculty of Informatics, \\ University of Debrecen, 
Pf.12, Debrecen, Hungary, H-4010; \\
\textit{e-mail}: \texttt{ispany.marton@inf.unideb.hu}, \\
$^3$MTA--SZTE Analysis and Stochastics Research Group, 
Bolyai Institute, \\
University of Szeged, 
Aradi v\'ertan\'uk tere 1, Szeged, Hungary, H--6720; \\
\textit{e-mail}: \texttt{kevei@math.u-szeged.hu} \\
$^4$Department of Stochastics, 
Bolyai Institute, University of Szeged, \\
Aradi v\'ertan\'uk tere 1, Szeged, Hungary, H--6720;\\
\textit{e-mail}: \texttt{papgy@math.u-szeged.hu}
\end{center}

\begin{abstract}
Multi-type inhomogeneous Galton--Watson process with immigration
is investigated, where the offspring mean matrix slowly converges
to a critical mean matrix.
Under general conditions we obtain limit distribution for the process,
where the coordinates of the limit vector are not necessarily independent.

\textit{Keywords:}
multi-type Galton--Watson process with immigration;
INAR model; nearly critical model;
multi-dimensional Poisson distribution; Toeplitz matrix;
inhomogeneous model.

\textit{AMS 2000 Subject Classification:} Primary 60J80, Secondary 60J27; 60J85 
\end{abstract}

\section{Introduction} \label{intr}

A zero start single-type inhomogeneous Galton--Watson branching process with
immigration (GWI process) $(X_n)_{n\in\ZZ_+}$ is defined as 
\[
\begin{cases}
X_n = \sum_{j=1}^{X_{n-1}} \xi_{n,j} + \vare_n ,
\qquad n \in \NN , \\
X_0 = 0 ,
\end{cases}
\]
where $\{ \xi_{n,j}, \, \vare_n : n, j \in \NN \}$
are independent random variables with non-negative integer values such
that for each $n\in\NN$, $\{ \xi_{n,j} : j \in \NN \}$ are identically
distributed.
We can interpret $X_n$ as the number of individuals in the $n^\mathrm{th}$
generation of a population, $\xi_{n,j}$ is the number of offsprings
produced by the $j^\mathrm{th}$ individual belonging to the
$(n-1)^\mathrm{th}$ generation, and $\vare_n$ is the number of
immigrants in the $n^\mathrm{th}$ generation.
A zero start one-dimensional inhomogeneous integer-valued autoregressive
(INAR) time series is a special single-type GWI process,
such that the offspring distributions are
Bernoulli.

Assume that $\varrho_n := \EE \xi_{n,1} < \infty$ and 
$m_n := \EE  \vare_n < \infty$.
A one-dimensional inhomogeneous GWI process $(X_n)_{n\in\ZZ_+}$ is called
\emph{nearly critical} if $\varrho_n\to 1$ as $n\to\infty$.
Gy\"orfi et al.~\cite{GyorIspPapVar} investigated the asymptotic behavior
of nearly critical one-dimensional INAR processes with $\varrho_n < 1$,
under the assumption $\sum_{n=1}^\infty (1-\varrho_n) = \infty$,
i.e.~the convergence $\varrho_n \to 1$, $n \to \infty$, is not too fast.
In the followings any non-specified limit relation is meant as
$n \to \infty$.
It turns out in Theorem 1 \cite{GyorIspPapVar}
that in case of Bernoulli immigration the process
$X_n$ converges in distribution to a Poisson distribution with parameter
$\lambda$, when $m_n / (1-\varrho_n) \to \lambda$. That is, if there is a
balance between the immigration $m_n$ and the extinction effect $1-\varrho_n$,
then a non-trivial limit distribution exists.
Moreover, in \cite{GyorIspPapVar} general immigration distributions are
investigated: when the factorial moments of the immigration at generation $n$
is of order $1- \varrho_n$ then compound Poisson limit appears.
These investigations were extended by Kevei \cite{K} for general GWI processes,
that is the Bernoulli assumption on the offsprings was relaxed.
In the present paper we investigate the multi-type version of the previous
problem.

In a multi-type homogeneous Galton--Watson process (without immigration) 
the main data of the process is the spectral radius $\varrho(\bB)$
of the mean matrix $\bB$, where
$\varr(\bB) := \max\{ |\lambda| : \text{$\lambda$ is an eigenvalue of $\bB$} \}$. 
By classical results, a positively regular, non-singular
multi-type Galton--Watson process dies out almost surely if and only if
$\varrho(\bB) \leq 1$. The process is called subcritical, critical
or supercritical if $\varrho < 1$, $=1$ or $>1$, respectively.
In the multi-type setup we also consider nearly-critical processes, that is
we assume that the sequence of offspring mean matrices converge to a
critical limit matrix. However, contrary to the one-dimensional case, there
are a lot of critical matrices, and thus a lot of nearly-critical processes.
The formal definition comes below.

An inhomogeneous multi-type GWI process with $d$ types
\[
\bX_n = ( X_{n,1} , \ldots, X_{n,d} ), \quad n \in \ZZ_+ ,
\]
defined as 
\[
\begin{cases}
\bX_n = \sum_{j=1}^{X_{n-1,1}} \bxi_{n,j,1}
+ \ldots + \sum_{j=1}^{X_{n-1,d}} \bxi_{n,j,d}
+ \bvare_n , \qquad n \in \NN , \\
\bX_0 = \bzero ,
\end{cases}
\]
where
$\{ \bxi_{n,j,i}, \, \bvare_n : n, j \in \NN , \, i \in \{1,\dots,d\} \}$
are independent $d$-dimensional random vectors with non-negative 
integer coordinates such that for each $n\in\NN$ and $i \in \{1,\dots,d\}$,
$\{ \bxi_{n,j,i} : j \in \NN \}$ are identically distributed, and
$\bzero$ is the zero vector.
Then $X_{n,i}$ is the number of $i$-type  individuals in the $n^\mathrm{th}$ generation
of a population, $\bxi_{n,j,i}$ is the number of offsprings produced by the $j^\mathrm{th}$
individual of type $i$ belonging to the
$(n-1)^\mathrm{th}$ generation, and $\bvare_n$ is the number of
immigrants.
When the offsprings are Bernoulli distributed 
(see Section \ref{section:general}
for the definition of multidimensional Bernoulli distribution)
we obtain the $d$-dimensional inhomogeneous INAR time series.

Suppose that the offspring and immigration means are finite.
Let us denote the offspring mean matrix and the immigration mean vector in the
$n^\mathrm{th}$ generation by
\[
\bB_n =
\begin{bmatrix}
\EE \bxi_{n,1,1} \\
\vdots \\
\EE \bxi_{n,1,d}
\end{bmatrix}
\in \RR_+^{d \times d}, \qquad
\EE \bvare_n = \bmm_n \in \RR_+^d ,
\]
where the elements of $\RR_+^d$ are $d$-dimensional row vectors with non-negative
coordinates.
Then $(\bB_n)_{i,j}$ is the expected number of type-$j$ offsprings of a single
type-$i$ particle in generation $n$. Then we have the recursion
\begin{equation} \label{rec}
\EE \bX_n = (\EE \bX_{n-1}) \,  \bB_n + \bmm_n , \qquad n \in \NN ,
\end{equation}
since
\begin{align*}
\EE( \bX_n \mid \bX_{n-1} )
&= \EE\Bigg( \sum_{j=1}^{X_{n-1,1}} \bxi_{n,j,1} + \ldots
+ \sum_{j=1}^{X_{n-1,d}} \bxi_{n,j,d}
+ \bvare_n \,\bigg|\, \bX_{n-1} \Bigg) \\
&= \sum_{j=1}^{X_{n-1,1}} \EE \bxi_{n,j,1} + \ldots
+ \sum_{j=1}^{X_{n-1,d}} \EE \bxi_{n,j,d}  + \EE \bvare_n \\
&= X_{n-1,1} \, \EE \bxi_{n,1,1}  + \ldots + X_{n-1,d} \, \EE \bxi_{n,1,d} 
+ \EE \bvare_n
= \bX_{n-1} \bB_n  + \bmm_n .
\end{align*}
The sequence $(\bB_n)_{n\in\NN}$ of the offspring mean
matrices plays a crucial role in the asymptotic behavior of the sequence
$(\bX_n)_{n\in\ZZ_+}$ as $n\to\infty$.
A $d$-dimensional inhomogeneous Galton--Watson process
$(\bX_n)_{n\in\ZZ_+}$ is
called \emph{nearly critical} if $\bB_n\to \bB$  and
$\varr(\bB) = 1$.
We will investigate the asymptotic behavior of nearly critical GWI processes.

Homogeneous multi-type GWI processes have been introduced and studied by
Quine \cite{Quine, Quine2}. In \cite{Quine} necessary and sufficient condition is 
given for the existence of stationary distribution in the subcritical case.
A complete answer is given by Kaplan \cite{kaplan}. Also Mode \cite{mode} gives
a sufficient condition for the existence of a stationary distribution, and 
in a special case he shows that the limiting distribution is
a multivariate Poisson with independent components.

Branching process models are extensively used in various parts of natural sciences,
among others in biology, epidemiology, physics, computer sciences.
In particular, multi-type GWI processes were used
to determine the asymptotic mean and covariance matrix of deleterious genes
and mutant genes in a stationary population by  Gladstien and Lange \cite{GlLa78},
and in non-stationary population by Lange and Fan \cite{LaFa97}.
Another rapidly developing area where multi-type GWI processes can be applied
is the theory of polling systems.
Resing \cite{Res} pointed out that a large variety of polling models can be
described as a multi-type GWI process.
Resing \cite{Res}, van der Mei \cite{Mei}, Boon \cite{boon1}, Boon et al.~\cite{boon2}
and Altman and Fiems \cite{alt} investigated several communication protocols applied in
info-communication networks with differentiated services.
There are different quality of services,
for example, some of them are delay sensitive
(telephone, on-line video, etc.), while others tolerate some delay (e-mail, internet,
downloading files, etc.). Thus, the services are grouped into service classes such
that each class has an own transmission protocol like priority queueing. In the papers
mentioned above the $d$-type Galton--Watson process has been used, where the process was
defined either by the sizes of the active user populations of the $d$ service classes,
or by the length of the $d$ priority queues.
For the general theory and applications of multi-type Galton--Watson processes
we refer to Mode \cite{mode} and Haccou et al.~\cite{HJV}.

The INAR time series as a particular case of GWI processes with Bernoulli offspring
distribution have been investigated by several authors, see e.g. the survey of
Wei\ss{} \cite{We08}. Heterogeneous INAR(1) models have been considered by
B\"ockenholt \cite{Bo99} for understanding and predicting consumers' buying
behavior, and Gourieroux and Jasiak \cite{GoJa04} for modeling the premium
in bonus-malus scheme of car insurance. Note that the higher order 
$\textrm{INAR}(p)$ times series introduced by Du and Li \cite{DuLi91} has
state space representation by a multivariate $\textrm{INAR}(1)$ model which
is a particular case of the multi-type GWI process, see Franke and Subba Rao
\cite{FrSu95}.

The paper is organized as follows.
In Section \ref{section:general} general sufficient conditions are
given for the mean matrices $\bB_n$
to get a non-trivial limit distribution for the sequence $\bX_n$.
In Section \ref{section:special} we spell out the general theorems
for some special cases of the mean matrices. We investigate here
the case when the limit matrix $\bB = \bI$, and when $\bB_n = \varrho_n \bB$.
The proofs are gathered in Section \ref{section:proofs}.

\section{General results} \label{section:general}

First we introduce some notation.
Boldface lower case letters $\bx, \by, \bk, \bell, \bmm, \blambda$
stand for $d$-dimensional
(row) vectors, boldface  upper case letters $\bA, \bB$ stand for
$d \times d$ real matrices, $(\bx)_i$ is the $i^\text{th}$ element of $\bx$,
$(\bA)_{i,j}$ is the element of $\bA$ in the $i^\text{th}$ row and $j^\text{th}$ column.
For the usual basis in $\RR^d$ we use the notation
\[
\be_1 = (1, 0, \ldots, 0 ), \ldots,
\be_d = (0,0, \ldots, 1),
\]
and for the constant zero and constant one vector we put
\[
\bzero =  (0, 0, \ldots, 0 ), \ 
\bbone =  (1, 1, \ldots, 1 ).
\]
Inequalities between vectors, and between matrices are meant elementwise.
For a vector $\bx \in \RR^d$ its norm is denoted by
$\| \bx \|$, where the norm is an arbitrary norm on the linear space $\RR^d$.
As an abuse of notation $\| \bA \|$ is the operator-norm of the matrix $\bA$,
induced by the norm $\| \cdot \|$ on the linear space $\RR^d$,
i.e.~$\| \bA \| = \sup_{\bx \, : \, \| \bx \| \leq 1} \| \bx \bA \|$.
Therefore, all the following statements are meant as: If there exist a
norm $\| \cdot \|$, such that the conditions of the statement hold with that
norm, then the conclusion holds. See the example after Proposition \ref{Xmean}.

The distribution of a random vector $\bxi$ will be denoted by
$\cL(\bxi)$.
For $\bp = (p_1, \ldots, p_d ) \in [0,1]^d$ with
$p_1 + \ldots + p_d \leq 1$, let $\Be(\bp)$  denote the
$d$-dimensional Bernoulli distribution with means $p_1, \dots, p_d$
defined by
\[
\Be(\bp)( \{ \be_1 \} ) = p_1 , \quad\dots,\quad
\Be(\bp)( \{ \be_d \} ) = p_d , \qquad
\Be(\bp)( \{ \bzero \} ) = 1 - p_1 - \ldots - p_d .
\]
If $\bxi = (\xi_1,\dots,\xi_d)$ is a random vector with
$\cL(\bxi) = \Be(\bp)$ then $\xi_1,\dots,\xi_d$ are random variables
with $\cL(\xi_i) = \Be(p_i)$, $i \in \{1,\dots,d\}$
(thus $\EE \bxi = \bp$), but $\xi_1,\dots,\xi_d$ are \emph{not} independent,
hence $\Be(\bp) \ne \Be(p_1) \times \ldots \times \Be(p_d)$.

When the offspring distributions are Bernoulli, each particle has
at most one offspring.
In this case $(\bX_n)_{n\in\ZZ_+}$ is an inhomogeneous INAR process, such
that $\cL(\bxi_{n,1,i}) = \Be(\be_i \bB_n)$. Note that in this case $\bB_n$
is substochastic matrix.

For $\blambda = (\lambda_1, \ldots, \lambda_d) \in [0, \infty)^d$, the
$d$-dimensional Poisson distribution with parameter $\blambda$ is
defined by
$\Po(\blambda) := \Po(\lambda_1) \times \ldots \times \Po(\lambda_d)$.
In other words, $\bxi = (\xi_1,\dots,\xi_d)$ is a random vector with
$\cL(\bxi) = \Po(\blambda)$ whenever $\xi_1,\dots,\xi_d$ are
independent random variables with $\cL(\xi_i) = \Po(\lambda_i)$,
$i \in \{1,\ldots,d\}$.
Note that $\EE \bxi = \blambda$.

Introduce the generating functions
\begin{equation} \label{eq:genfc-def}
\begin{split}
& F_n(\bx) = \EE \bx^{\bX_n}, \quad G_{n,i}(\bx) = \EE \bx^{\bxi_{n,1,i}}, \quad
\bG_n(\bx) = ( G_{n,1}(\bx), \ldots, G_{n,d}(\bx) ), \\
& H_n(\bx) = \EE \bx^{\bvare_n}, \quad \bx \in [0,1]^d,
\end{split}
\end{equation}
where $\bx^{\bk} = x_1^{k_1} \ldots x_d^{k_d}$.
Conditioning argument shows
the recursion $F_n(\bx) = F_{n-1}(\bG_n(\bx)) H_n(\bx)$, $n \geq 2$.
Let denote $\overline \bG_{n+1,n}(\bx) = \bx$, and if $\overline \bG_{j+1,n}$
is defined then $\overline \bG_{j,n}(\bx) =  \bG_j (\overline \bG_{j+1,n}(\bx))$.
With this notation Quine \cite{Quine} proved 
(simple induction argument shows) that we have
\begin{equation} \label{eq:genfc-F}
F_n(\bx) = \prod_{j=1}^n H_j ( \overline \bG_{j+1,n}(\bx)).
\end{equation}

It turns out that due to the near-criticality under general
conditions
$H_j ( \overline \bG_{j+1,n}(\bx)) \approx 1$, for each $j$, thus
\[
H_j ( \overline \bG_{j+1,n}(\bx)) \approx 
\exp \left\{  H_j ( \overline \bG_{j+1,n}(\bx) ) - 1  \right\},
\]
therefore
it is reasonable to define the accompanying compound Poisson probability generating function
\begin{equation} \label{eq:acc-Poi}
\widetilde F_n(\bx) = \exp 
\left\{
\sum_{j=1}^n 
\left[ H_j ( \overline \bG_{j+1,n}(\bx)) - 1 \right] \right\}.
\end{equation}
We prove in Lemma \ref{lemma:F-tildeF} that under some conditions
\[
\lim_{n\to \infty} ( F_n(\bx) - \widetilde F_n(\bx) ) = 0.
\]
Therefore to determine the asymptotic properties of $\bX_n$ we have to
investigate the sum
\[
\sum_{j=1}^n \left[ H_j (\overline \bG_{j+1,n}(\bx )) - 1 \right].
\]

We can compute explicitly the generating function when both the immigration
and the offsprings have Bernoulli distribution. Indeed,
when ${\cal L}(\bvare_n) = \textrm{Be}(\bmm_n)$, the immigration generating function is
\[
H_n ( \bx) = \EE \bx^{\bvare_n} = m_{n,1} x_1 + \ldots + m_{n,d} x_d
+ 1 - (m_{n,1} + \ldots + m_{n,d}) = 1 + (\bx - \bbone) \bmm_n ^\top;
\]
when the offspring distributions are also Bernoulli then
$\bG_n(\bx) = \bbone + (\bx - \bbone) \bB_n^\top$, and so
$\overline \bG_{j+1,n}(\bx) = \bbone + (\bx - \bbone) \bB_{[j,n]}^\top$,
where
\[
\bB_{[j,n]} := \begin{cases}
\bB_{j+1} \ldots \bB_n,
& \text{for $0 \leq j \leq n-1$,} \\
\bI, & \text{for $j=n$}.
\end{cases}
\]
Note that in this paper the multivariate Bernoulli distribution is defined in
such a way that its generating function is a first order polynomial which is
a particular case of a more general definition of the multivariate Bernoulli
distribution, see Krummenauer \cite[Definition 1]{Kr98}.
Thus (\ref{eq:acc-Poi}) reads as
\begin{equation*} \label{eq:F_n-form-Be-imm}
\widetilde F_n ( \bx ) = \exp \left\{ 
\sum_{j=1}^n \left[ (\bx - \bbone) \bB_{[j,n]}^\top \bmm_j^\top \right]
\right\}.
\end{equation*}

Observe that the recursion \eqref{rec} implies
\begin{equation}\label{EX_n}
\EE \bX_n = \sum_{j=1}^n \bmm_j \, \bB_{[j,n]} .
\end{equation}
This can be obtained also by differentiating $F_n$ in (\ref{eq:genfc-F}).
Putting
\begin{equation} \label{eq:A-def}
\bA_{j,n} = ( \bB - \bB_j ) \bB_{[j,n]},
\qquad n \in \NN , \quad j \in \{ 1, \dots, n \},
\end{equation}
we may rewrite (\ref{EX_n}) as
\[
\EE \bX_n = \sum_{j=1}^n  \bmm_j (\bB - \bB_j)^{-1}  \bA_{j,n},
\]
whenever the inverse $(\bB - \bB_j)^{-1}$ exists for each $j=1,\ldots, n$.

These computations shows the necessity for  a summability
method defined by the offspring mean matrices.
We will make of use of the following matrix version of Toeplitz theorem
(see, e.g., in Fritz \cite{Fritz}).
\begin{lemma} \label{Toeplitz}
Let $\bA_{j,n} \in \RR^{d \times d}$, $n \in \NN$, $j=1,2, \ldots, n$ be matrices such that
\begin{eqnarray}
&& \lim_{ n \to \infty} 
\max_{1 \leq j \leq n} \|\bA_{j,n}\| = 0,  \label{T1} \\
&& \lim_{ n \to \infty} \sum_{j=1}^n \bA_{j,n} = \bA, \label{T2} \\
&& \sup_{n \in \NN} \sum_{j=1}^n \| \bA_{j,n} \| < \infty. \label{T3}
\end{eqnarray}
Then for any convergent sequence of vectors $\bx_n \to \bx$
$$
\sum_{j=1}^n  \bx_j \bA_{j,n} \to \bx \bA .
$$
\end{lemma}

In fact, the Lemma holds with the weaker assumption
$ \| \bA_{j,n} \| \to 0$, for all $j \in \NN$, instead of (\ref{T1}). However,
for the proof of Lemma \ref{lemma:F-tildeF}
the stronger version is needed.

\begin{lemma} \label{lemma:mean-toeplitz}
Assume that the sequence of mean matrices $( \bB_n )_{n \in \NN}$ satisfies the following
conditions:
\begin{itemize}
\item[\textnormal{(B1)}] $\lim_{n \to \infty} \bB_n = \bB$,
for some 
matrix $\bB$;
\item[\textnormal{(B2)}] $\| \bB_n \| \leq 1$ and $\bB - \bB_n$ is invertible
whenever $n \geq n_0$ for some $n_0$;
\item[\textnormal{(B3)}] $\lim_{n \to \infty} \| \bB_{[j,n]} \| = 0$ for any fixed $j$;
\item[\textnormal{(B4)}] $\lim_{n \to \infty} \sum_{j=1}^n (\bB - \bB_j) \bB_{[j,n]}  = \bA$
for some limit matrix $\bA$;
\item[\textnormal{(B5)}] $\sup_n \sum_{j=1}^n \|(\bB- \bB_j)  \bB_{[j,n]} \| < \infty$.
\end{itemize}
Then the triangular matrix array $( \bA_{j,n} = (\bB - \bB_j) \bB_{[j,n]} )_{j,n}$ satisfies
the conditions of Lemma \ref{Toeplitz}.
\end{lemma}

The following two general theorems give sufficient condition for the convergence
of $\bX_n$. It turns out that in case of Bernoulli offspring and immigration
only conditions (\ref{T1})--(\ref{T3}) have to be assured.
Note that when the offspring distribution is Bernoulli, then the limit matrix $\bB$
in \textnormal{(B1)} is necessarily substochastic.

\begin{theorem} \label{th:INAR}
Let $(\bX_n)_{n \in \ZZ_+}$ be an inhomogeneous GWI process such that
both the offspring and the immigration have Bernoulli distribution and
\textnormal{(B1)}--\textnormal{(B5)} hold.
If
\begin{itemize}
\item[\textnormal{(M)}] $\lim_{n\to \infty} \bmm_n (\bB - \bB_n)^{-1} = \blambda$,
\end{itemize}
then
\[
\bX_n \stackrel{\mathcal{D}}{\longrightarrow} \Po(\blambda \bA),
\]
where $\bA$ is given in \textnormal{(B4)}.
\end{theorem}

The Bernoulli distribution of the
offsprings and the immigration is a very restrictive condition.
In the following theorems we  weaken these assumptions.

The interesting new feature in the following theorem is that the 
components of the limit are dependent in general. We need some further
notation.

For a multi-index $\bj = (j_1,\dots,j_d) \in \ZZ_+^d$
let denote $m_{n,\bj}$ the $\bj^\text{th}$ factorial moment of the immigration
$\bvare_n =(\vare_{n,1}, \ldots, \vare_{n,d})$, that is
\[
m_{n,\bj} = 
\EE\left( \prod_{i=1}^d \vare_{n,i} (\vare_{n,i}-1)\ldots(\vare_{n,i}-j_i+1) \right)
= D^{\bj} H_n(\bbone),
\]
where for a multi-index $\bj$
\[
D^{\bj} H_n(\bbone) = \frac{\partial^{|\bj|}}{\partial^{j_1} x_1 \ldots \partial^{j_d} x_d}
H_n(\bbone),
\]
$|\bj| = j_1 + \ldots + j_d$, and the derivatives are meant as the left derivatives.

We cannot circumvent the fairly inconvenient notation below and in Lemma \ref{Taylor}, because
formulas (\ref{eq:imm-assumption-II}) and (\ref{eq:limit-genfc}) are not easily
translated to the multi-index notation.

\begin{theorem} \label{th:general-immigration}
Let $(\bX_n)_{n \in \ZZ_+}$ be an inhomogeneous GWI process with 
Bernoulli offspring distributions, such that
\textnormal{(B1)}--\textnormal{(B5)} hold.
Moreover, assume that for some
$k \geq 2$
\begin{equation} \label{eq:imm-assumption-I}
\lim_{n\to\infty} 
\|(\bB - \bB_n)^{-1}\| \, \max_{|\bj|= k} D^{\bj} H_n(\bbone)  = 0,
\end{equation}
and for each $i=1,2, \ldots, k-1$,
for each $1 \leq \ell_{i+1}, \ldots, \ell_{2i} \leq d$, the limit
\begin{equation} \label{eq:imm-assumption-II}
\lim_{n \to \infty}
\sum_{j=1}^n \frac{1}{i!}
\sum_{\ell_1, \ldots, \ell_i=1}^d
\frac{\partial^i H_j(\bbone)}
{\partial x_{\ell_1} \ldots \partial x_{\ell_i}}
\left( \bB_{[j,n]} \right)_{\ell_1, \ell_{i+1}}
\ldots  \left( \bB_{[j,n]} \right)_{\ell_{i}, \ell_{2i}}
=: \Lambda_{i; \ell_{i+1}, \ldots, \ell_{2i}}
\end{equation}
exists. Then
\[
\bX_n \stackrel{\mathcal{D}}{\longrightarrow} \bY,
\]
where
\begin{equation} \label{eq:limit-genfc}
\EE \bx^{\bY} = \exp \left\{
\sum_{i=1}^{k-1} 
\sum_{\ell_{i+1}, \ldots, \ell_{2i} = 1}^d
\Lambda_{i; \ell_{i+1}, \ldots, \ell_{2i}}
(x_{\ell_{i+1}} - 1) \ldots (x_{\ell_{2i}} - 1)
\right\}.
\end{equation}
\end{theorem}

Note that if (\ref{eq:imm-assumption-II}) holds, then necessarily
$\Lambda_{i;\cdot}$ is symmetric in the sense that for any permutation
$\pi$ we have $\Lambda_{i;\ell_1, \ell_2, \ldots, \ell_i} =
\Lambda_{i; \ell_{\pi 1}, \ell_{\pi 2}, \ldots,  \ell_{\pi i}}$. In particular,
$\Lambda_{2; j,k } = \Lambda_{2; k,j}$, we use this in Example \ref{example:1}.

A simple sufficient condition which guarantees
(\ref{eq:imm-assumption-I}) is that there are at most
$k-1$ immigrants in any generation. The other condition is more difficult to
check, however for $i=1, 2$ we can write it in a simpler form.

For $i=1$ condition (\ref{eq:imm-assumption-II}) is just
the convergence
\[
\sum_{j=1}^n \bmm_j  \bB_{[j,n]} \to \blambda = (\lambda_1, \ldots, \lambda_d), 
\]
with $\Lambda_{1;\ell} = \lambda_\ell$. Since the matrix array $( \bA_{j,n} )$
satisfies the conditions of Lemma \ref{Toeplitz}
we see that the convergence above follows from
condition (M). As a consequence we obtain

\begin{corollary} \label{cor:spec-imm}
Let $(\bX_n)_{n \in \ZZ_+}$ be an inhomogeneous GWI process with 
Bernoulli offspring distributions, such that
\textnormal{(B1)}--\textnormal{(B5)} hold.
Moreover, assume \textnormal{(M)} and
\begin{equation} \label{eq:imm-assumption}
\lim_{n\to\infty} \EE \|\bvare_n\|^2  \, \|(\bB - \bB_n)^{-1}\| = 0 .
\end{equation}
Then
\[
\bX_n \stackrel{\mathcal{D}}{\longrightarrow} \Po(\blambda \bA).
\]
\end{corollary}

For $i=2$ condition (\ref{eq:imm-assumption-II}) takes the form
\[
\frac{1}{2} \sum_{j=1}^n \bB_{[j,n]}^\top \Delta_j \bB_{[j,n]} \to \bLambda_2,
\]
with $\Lambda_{2; k, \ell} = (\bLambda_2)_{k, \ell}$,
where $\Delta_j = \Delta H_j(\bbone)$ is the Hesse-matrix of the immigration
generating function at $\bbone$; i.e.~$(\Delta H_j(\bbone) )_{k, \ell}
=\frac{\partial^2}{\partial x_k \partial x_\ell} H_j(\bbone)$.

\begin{example} \label{example:1}
The following simple example shows that the limit may have dependent components
even in a simple case. Let $d=2$, ${\cal L}(\bxi_{n,1,i}) = \textrm{Be}((1-n^{-1})
\be_i)$, $i=1,2$, and $\PP(\bvare_n=\bzero)=1-n^{-1}$, $\PP(\bvare_n=\bbone)=n^{-1}$
for all $n\in\NN$, that is in the $n^\text{th}$ generation each particle survives with 
probability $1- n^{-1}$, and with probability $n^{-1}$ a type-1 and a type-2 
particle immigrate together. Then we have
\[
\bB_n = \left(1- \frac{1}{n} \right) \bI, \quad \text{and} \quad
H_n(x_1, x_2) = 1 - \frac{1}{n} + \frac{x_1 x_2}{n}.
\]

Clearly condition (\ref{eq:imm-assumption-I}) holds with $k=3$.
The relevant quantities are $\bB=\bI$, $\bmm_n = \frac{1}{n} (1,1)$,
\[
\Delta_n = \frac{1}{n}
\begin{bmatrix}
0 & 1 \\ 1 & 0
\end{bmatrix},
\quad
\bmm_n (\bB - \bB_n )^{-1} = \frac{1}{n} (1,1) n \bI  = (1,1),
\]
and
\[
\sum_{j=1}^n \bB_{[j,n]}^\top \Delta_j \bB_{[j,n]} =
\sum_{j=1}^n \frac{j}{n^2} \begin{bmatrix}
0 & 1 \\ 1 & 0
\end{bmatrix} \to \frac{1}{2} \begin{bmatrix}
0 & 1 \\ 1 & 0
\end{bmatrix}.
\]
We see that $\Lambda_{1;1} = \Lambda_{1;2} = 1$ and $\Lambda_{2;1,2} = \Lambda_{2;2,1}= 1/4$,
$\Lambda_{2; 1,1} = \Lambda_{2;2,2} = 0$.
Thus
\[
\bX_n \stackrel{\mathcal{D}}{\longrightarrow} \bY, \quad
\text{ where } \quad
\EE \bx^{\bY} = \exp \left\{ x_1 -1 + x_2 -1 + \frac{(x_1 -1)(x_2 -1)}{2} 
\right\}.
\]
Let $U, V, W$ be independent Poisson random variables with parameters
$\lambda_1, \lambda_2, \mu$, respectively.
The generating function of $(U+W, V+W)$ is given by
\[
\begin{split}
\EE x_1^{U+ W} x_2^{V + W}
& = \EE x_1^U \EE x_2^V \EE (x_1 x_2)^W \\
& =
\exp \left\{ \lambda_1 (x_1 -1) + \lambda_2(x_2 -1) + \mu (x_1 x_2-1)\right\},
\end{split}
\]
therefore the distribution of the limit $Y$ is the distribution of the vector
$(U+W, V + W)$ where $U, V, W$ are iid Poisson($1/2$).
The distribution is called bivariate Poisson distribution,
with parameters $\lambda_1$, $\lambda_2$, and $\mu$, 
see Johnson et al. \cite{JKB} p.124, or
Kocherlakota et al. \cite{KoKo92}.
\end{example}

In general, when in the exponent in (\ref{eq:limit-genfc})
none of the terms are divisible with $x_i^2$ for any $i$
(e.g.~at most 1 particle immigrates for any given type),
then the components of the limit $Y$ in Theorem \ref{th:general-immigration}
can be represented as the sum of
independent Poisson random variables. Assume that the conditions of
Theorem \ref{th:general-immigration} hold, with $k=3$ in (\ref{eq:imm-assumption-I}),
and for the limits in (\ref{eq:imm-assumption-II})
$\Lambda_{2;i,i} = 0$ for all $i$. Then the limit random vector
$\bY = (Y_1, \ldots, Y_d)$ can be represented as
\[
Y_i= U_i + \sum_{j \ne i} U_{i,j}, \quad i=1,\ldots, d,
\]
where $( U_i )_{i=1}^d$ and $( U_{i,j} )_{1 \leq i < j \leq d}$ are
independent Poisson random variables, with parameters $a_i$ and $a_{i,j}$
respectively, with
\[
a_i = \Lambda_{1;i} - 2 \sum_{j=1}^d \Lambda_{2;i,j}, \ \ i=1, \ldots, d, 
\quad
a_{i,j} = 2 \Lambda_{2;i,j}, \ \ 1 \leq i < j \leq d,
\]
and $U_{i,j} := U_{j,i}$ for $i > j$. 
It is not difficult to show that (\ref{eq:imm-assumption-I}) with $k=3$ and
(\ref{eq:imm-assumption-II}) imply that the coefficients above are non-negative.
Simple computation shows that the 
generating function of $\bY$ agrees with the one given in Theorem \ref{th:general-immigration}.
Clearly, this construction extends for $k \geq 3$. The appearing limiting distributions
are the so-called multivariate Poisson distributions; for further properties
see \cite{JKB}, p.139. Note that this multivariate Poisson distribution
appears as a limit in the multivariate version of the law of small numbers, see
Krummenauer \cite[Theorem 1]{Kr98}. Hence, Theorem \ref{th:general-immigration}
can be interpreted as a general law of small numbers for inhomogeneous
GWI processes.
Also note the difference between the multivariate Poisson distribution introduced here
and the $d$-dimensional Poisson distribution defined before (\ref{eq:genfc-def}).

In the next theorem the condition on the offspring distribution is
relaxed, though (\ref{eq:offs-assumption})
means that the offspring distribution
has to be very close to a Bernoulli distribution.
Note that in this case we assume that the limit matrix is
the unit matrix $\bI$, in which case condition \textnormal{(B4)}
automatically holds, with limit matrix $\bA = \bI$.
We return to this question in Subsection 3.1.
To state the theorem we introduce the notation
\begin{equation} \label{eq:def-m2}
m_2(n) = \max_{1 \leq i,j,l \leq d}
\frac{\partial^2}{\partial x_j \partial x_l}
G_{n,i}(\bbone).
\end{equation}

\begin{theorem} \label{th:general-offspring}
Let $(\bX_n)_{n \in \ZZ_+}$ be an inhomogeneous GWI process with 
Bernoulli immigration, such that $\bB= \bI$ and 
\textnormal{(B1)}--\textnormal{(B5)} hold.
Moreover, assume \textnormal{(M)} and
\begin{equation} \label{eq:offs-assumption}
\lim_{n \to \infty} m_2(n)  \|(\bI - \bB_n)^{-1} \| = 0.
\end{equation}
Then
\[
\bX_n \stackrel{\mathcal{D}}{\longrightarrow} \Po(\blambda).
\]
\end{theorem}

In the single-type case general immigration distribution is investigated
and convergence to a compound Poisson distribution is proved in
Theorem 4 in \cite{GyorIspPapVar}, and in Theorem 3 in \cite{K}.
In case of more general offspring distribution existence of
negative binomial limit is showed in Theorem 5 in \cite{K}.
However, in our multi-type scenario the computations with general
immigration or (and) with general offspring distribution become
intractably complicated.

Finally, we note that if $\prod_{n=1}^\infty \bB_n$ exists and is
not the $0$ matrix, then the process $\bX_n$ converges when
$\sum_{n=1}^\infty \bmm_n$ is finite. This case can be handled 
similarly as in the single-type scenario in \cite{K}.

\section{Special cases and examples} \label{section:special}

In what follows we investigate some special cases for the
sequence of mean matrices,
and we give sufficient conditions for the existence of the distributional
limit, which are easier to handle than the ones given in
Theorem \ref{th:INAR}.

\subsection{The case $\bB = \bI$}

When the critical limit matrix is the identity matrix then
one of  the most complicated assumption,
(B4) in Theorem \ref{th:INAR}, holds automatically.
In this case $\bA_{j,n} = \bB_{[j,n]} - \bB_{[j-1,n]}$, and so
$\sum_{j=1}^{n} \bA_{j,n}$ is a telescopic sum.

\begin{proposition}\label{Xmean}
Suppose that
\begin{itemize}
\item[\textnormal{(I1)}] 
$\lim_{n\to \infty} \bB_n = \bI$;
\item[\textnormal{(I2)}] 
there is an $n_0$ such that $\|\bB_n\| < 1$ for all $n \geq n_0$;
\item[\textnormal{(I3)}] 
$\lim_{n\to \infty} \| \bB_{[j,n]} \| = 0$ for all $j \in \NN$;
\item[\textnormal{(I4)}] 
there is an $n_0$ such that
$\sup\limits_{n \geq n_0} \frac{\| \bI - \bB_n \|}{1 - \|\bB_n\|}
< \infty$ or $\bA_{j,n} \in \RR_+^{d \times d}$
for all $n \geq n_0$ and all $j \in \{1, \dots, n\}$.
\end{itemize}
Then the triangular matrix array $( \bA_{j,n} )_{j,n}$ satisfies
the conditions of Lemma \ref{Toeplitz}.
\end{proposition}

Note that condition \textnormal{(I2)} guarantees the existence of the inverse
$( \bI - \bB_n )^{-1}$ in (M) for $n \geq n_0$.

As we mentioned, the norm can be arbitrary operator norm. It is easy to
construct examples, such that some conditions hold in one operator norm, and
fail in another. For instance, let
\[
\bB_n =
\begin{bmatrix}
1- \frac{1}{n} & \frac{1}{n} \\
0 & 1- \frac{2}{n}
\end{bmatrix}.
\]
Then in column sum norm (induced by the $\ell_\infty$ norm on $\RR^d$)
condition \textnormal{(I2)} and \textnormal{(I4)} hold, while in the 
row sum norm (induced by the $\ell_1$ norm on $\RR^d$) even condition
\textnormal{(I2)} fails.

For jointly diagonalizable offspring mean matrices a better result is
available, namely, condition (I4) above can be omitted.

\begin{proposition}\label{Xmean_diag}
Suppose that  conditions \textnormal{(I1)}--\textnormal{(I3)} of
Proposition \ref{Xmean} hold, and the offspring mean matrices are of the form
\[
\bB_n = \bU \diag(\varrho_{n,1}, \dots, \varrho_{n,d}) \, \bU^\top,
\qquad n \in \NN ,
\]
where $\bU \in \RR^{d \times d}$ is an orthogonal matrix.
Then the triangular matrix array $( \bA_{j,n} )_{j,n}$ satisfies
the conditions of Lemma \ref{Toeplitz}.
\end{proposition}

As a consequence we obtain that the corresponding versions of Theorem
\ref{th:INAR} and \ref{th:general-immigration} can be stated.
For example the following holds.

\begin{theorem}\label{th:spec_I}
Let $(\bX_n)_{n\in\ZZ_+}$ be an inhomogeneous GWI process with
Bernoulli offspring and immigration distributions.
Assume that either
conditions of Proposition \ref{Xmean} or Proposition \ref{Xmean_diag} are
satisfied, and for the immigration \textnormal{(M)} holds. Then 
\[
\bX_n \distr \Po(\blambda).
\]
\end{theorem}

\begin{remark}\label{special}
The statement of Proposition \ref{Xmean_diag} for the special case
$\bB_n = \varrho_n \bI$, $n \in \NN$, with $\varrho_n \in [0,1]$, $n \in \NN$,
also follows from Proposition \ref{Xmean}; in this case
Theorem \ref{th:spec_I} imply the appropriate results for
one-dimensional inhomogeneous INAR processes due to
Gy\"orfi et al.~\cite{GyorIspPapVar}.
\end{remark}

Note that under the assumption of Proposition \ref{Xmean_diag} conditions
\textnormal{(I1)}--\textnormal{(I3)}
of Proposition \ref{Xmean} are equivalent to
\begin{itemize}
\item[\textnormal{(I1')}] 
$\lim\limits_{n\to \infty} \varrho_{n,i} = 1$ for all
$i\in\{1,\dots,d\}$;
\item[\textnormal{(I2')}] 
$\max\limits_{1 \leq i \leq d} \varrho_{n,i} < 1$ for all $n \geq n_0$;
\item[\textnormal{(I3')}] 
$\prod\limits_{n=j}^\infty \varrho_{n,i} = 0$ for all $j \in \NN$
and all  $i\in\{1,\dots,d\}$,
\end{itemize}
respectively.
Remark that conditions \textnormal{(I3)} and \textnormal{(I3')} are also equivalent to
\begin{itemize}
\item[\textnormal{(I3'')}] 
$\sum\limits_{n=1}^\infty (1-\varrho_{n,i}) = +\infty$ for all
$i\in\{1,\dots,d\}$.
\end{itemize}

The following example shows that Proposition \ref{Xmean_diag} can really
perform better for jointly diagonalizable offspring mean matrices than Proposition
\ref{Xmean}.

\begin{example}
Let $d=2$, $\varrho_{n,1}= 1 - 1/n$, $\varrho_{n,2}= 1- 1/\sqrt{n}$,
\[
\bU = \frac{1}{\sqrt{2}}
\begin{bmatrix} 1 & -1 \\ 1 & 1 \end{bmatrix}, \ \text{hence }
\bB_n = \begin{bmatrix}
1-\frac{\sqrt{n}+1}{2n} & \frac{\sqrt{n}-1}{2n} \\
\frac{\sqrt{n}-1}{2n} & 1-\frac{\sqrt{n}+1}{2n}
\end{bmatrix}.
\]
Then conditions \textnormal{(I1)}--\textnormal{(I3)} of 
Proposition \ref{Xmean} are trivially
satisfied, but condition (I4) of Proposition \ref{Xmean} fails to hold.
Indeed,
\[ 
\bA_{n,n} = \bI - \bB_n
= \begin{bmatrix}
\frac{\sqrt{n}+1}{2n} & -\frac{\sqrt{n}-1}{2n} \\
-\frac{\sqrt{n}-1}{2n} & \frac{\sqrt{n}+1}{2n}
\end{bmatrix}\notin\RR_+^{2\times2}, \quad
\|\bB_n\|=\left\| \diag\left(1-\frac{1}{n},1-\frac{1}{\sqrt{n}} \right)\right\|
=1-\frac{1}{n}
\]
and
\[
\begin{split}
\|\bI-\bB_n\|
& = \left\|\bU \left( \bI- \diag\left( 1-\frac{1}{n},1-\frac{1}{\sqrt{n}}\right) \right)
\bU^\top \right\|
=\left\| \bI-\diag\left( 1-\frac{1}{n},1-\frac{1}{\sqrt{n}} \right) \right\| \\
& = \left\| \diag \left( \frac{1}{n},\frac{1}{\sqrt{n}} \right) \right\|
=\frac{1}{\sqrt{n}}
\end{split}
\]
imply
$\sup\limits_{n \geq n_0} \frac{\| \bI - \bB_n \|}{1 - \|\bB_n\|}
= \infty$.
Here we used the simple fact that the norm of a normal element in a $C^*$-algebra
is equal to its spectral radius.
\end{example}

\subsection{The case $\bB_n = \varrho_n \bB$}

In this subsection we assume that
$\bB_n = \varrho_n \bB$, for all $n \in \NN$,
where $\bB$ is a substochastic matrix, and $\varrho_n < 1, \varrho_n \to 1$ and
$\sum_{n=1}^\infty (1 - \varrho_n) = \infty$.
In this special case $\bB_{[j,n]} = \varrho_{[j,n]} \bB^{n-j}$, with
$\varrho_{[j,n]}=\varrho_{j+1} \ldots \varrho_n$.
Put $a_{j,n}= \varrho_{[j,n]} ( 1 - \varrho_j)$, then
$\bA_{j,n} = \bB_{[j,n]} ( \bB - \bB_j)= a_{j,n} \bB^{n-j+1}.$

To apply Theorem \ref{th:INAR} or \ref{th:general-immigration} in this case, the 
missing condition is again \textnormal{(B4)}.
In the following statement we give a rather general condition for the existence
of the limit matrix. The key point is a slight modification of the proof of
Theorem 5.2.1 in Doob \cite{doob}.

\begin{proposition} \label{prop:doob}
Let $( \varrho_n )_{n \in \NN}$ be  a sequence, such that $\varrho_n < 1$, $\varrho_n \to 1$,
$\sum_{n=1}^\infty ( 1 - \varrho_n) = \infty$ and $(1 - \varrho_n)/ ( 1 - \varrho_{n+1}) \to 1$.
Then  for any matrix $\bB$ such that $\|\bB\| \leq 1$  the limit
$$
\lim_{n \to \infty} \sum_{k=1}^n \varrho_{[k,n]} ( 1 - \varrho_k) \bB^{n-k} = \bA
$$
exists, and $ \bB \bA = \bA \bB = \bA = \bA^2$.
\end{proposition}

It will be clear from the proof that whenever $\bB$ is stochastic the limit $\bA$ is
stochastic too.

Note that in the single-type case no additional assumption is needed
on the sequence $( \varrho_n )_{n \in \NN}$, see \cite{GyorIspPapVar} or \cite{K}.
Indeed, the condition $\sum_{n=1}^\infty (1- \varrho_n) = \infty$ implies that the 
numerical triangular array $( a_{j,n} = \varrho_{[j,n]}(1-\varrho_j) )$ satisfies
the following conditions, which are the 1-dimensional analog of (\ref{T1}), (\ref{T2})
and (\ref{T3}):
\begin{equation} \label{eq:1d-toeplitz}
\begin{split}
& \lim_{n \to \infty} \max_{1 \leq j \leq n} a_{j,n} =0, \\
& \lim_{n \to \infty} \sum_{j=1}^n a_{j,n} = 1, \\
& \sup_{n \geq 1} \sum_{j=1}^n | a_{j,n} | < \infty.
\end{split}
\end{equation}

The following example shows that when dealing with matrices the
additional assumption is in fact necessary.

\begin{example}
Let
\[
\varrho_n =
\begin{cases}
1, & \text{if $n$ is odd,}\\
1 - \frac{2}{n}, & \text{if $n$ is even,}
\end{cases}
\]
and
\[
\bB= \begin{bmatrix} 0 & 1 \\ 1 & 0 \end{bmatrix}.
\]
Then $\varrho_{[k,n]}=[k/2] / [n/2]$, where $[ \cdot ]$ stands for the (lower)
integer part, and so
\[
\varrho_{[k,n]} ( 1 - \varrho_k) =
\begin{cases}
0, & \text{if $k$ is odd,}\\
\frac{1}{[n/2]}, & \text{if $k$ is even.}
\end{cases}
\]
Thus we obtain
\[
\sum_{k=1}^{2n} \bA_{k,2n} = \sum_{k=1}^{2n} \bB^{2n-k+1} \varrho_{[k,2n]} (1- \varrho_k)
\to
\begin{bmatrix}
0 & 1 \\ 1 & 0
\end{bmatrix},
\]
while
\[
\sum_{k=1}^{2n+1} \bA_{k,2n+1} = \sum_{k=1}^{2n+1} \bB^{2n-k+2} \varrho_{[k,2n+1]} (1- \varrho_k)
\to
\begin{bmatrix}
1 & 0 \\ 0 & 1
\end{bmatrix},
\]
thus the limit does not exist.
\end{example}

Using Proposition \ref{prop:doob} and Theorem \ref{th:general-immigration}
we obtain the following

\begin{theorem} \label{th:rho_n}
Assume that the mean matrix of the Bernoulli offspring distribution has the form
\[
\bB_n = \varrho_n \bB,
\]
where $\bB$ is an invertible substochastic matrix, and $\varrho_n < 1$, $\varrho_n \to 1$,
$\sum_{k=1}^\infty (1 - \varrho_n) = \infty$  and $(1 - \varrho_n)/ ( 1 - \varrho_{n+1}) \to 1$.
Moreover, assume that for some $k \geq 2$
\[
\lim_{n\to\infty} 
\frac{ \max_{|\bj|= k} D^{\bj} H_n(\bbone)}{1- \varrho_n}  = 0, \quad
\lim_{n \to \infty} \frac{\bmm_n}{1-\varrho_n} = \blambda,
\]
and for each $i=2, \ldots, k-1$, for each
$1 \leq \ell_{i+1}, \ldots, \ell_{2i} \leq d$ the limit
\[
\lim_{n \to \infty}
\sum_{j=1}^n \frac{(\varrho_{[j,n]})^i}{i!}
\sum_{\ell_1, \ldots, \ell_i=1}^d
\frac{\partial^i H_j(\bbone)}
{\partial x_{\ell_1} \ldots \partial x_{\ell_i}}
\left( \bB^{n-j} \right)_{\ell_{1}, \ell_{i+1}}
\ldots  \left( \bB^{n-j} \right)_{\ell_{i}, \ell_{2i}}
=: \Lambda_{i; \ell_{i+1}, \ldots, \ell_{2i}}
\]
exists. Then
\[
\bX_n \stackrel{\mathcal{D}}{\longrightarrow} \bY,
\]
where
\[
\EE \bx^{\bY} = \exp \left\{
(\bx - \bbone ) (\blambda \bA)^\top  +
\sum_{i=2}^{k-1} 
\sum_{\ell_{i+1}, \ldots, \ell_{2i} = 1}^d
\Lambda_{i; \ell_{i+1}, \ldots, \ell_{2i}}
(x_{\ell_{i+1}} - 1) \ldots (x_{\ell_{2i}} - 1)
\right\},
\]
where the matrix $\bA$ is given by Proposition \ref{prop:doob}.
\end{theorem}

The next theorem gives more freedom on the mean matrix $\bB_n$, however
stronger assumption on the limit matrix $\bB$ is needed.
The inequality for the matrices are meant elementwise.

\begin{theorem} \label{th:spec-2}
Assume that for the mean matrix of the Bernoulli offspring distribution
\[
\vartheta_n \bB \leq \bB_n \leq \varrho_n \bB
\]
holds, where $\vartheta_n \leq \varrho_n < 1$, $\vartheta_n \to 1$, $\varrho_n \to 1$,
$\sum_{n=1}^\infty ( 1 - \varrho_n) = \infty$, and
$(\varrho_n - \vartheta_n )/ ( 1 - \varrho_n) \to 0$, and
for the immigration
\[
\bmm_n (\bB - \bB_n)^{-1}  \to \blambda.
\]
If either (a) $\bB^n \to \bA$ for some matrix $\bA$ or (b)
$\|\bB\| \leq 1$ and $(1- \varrho_{n+1} ) / ( 1 - \varrho_n) \to 1$ holds  then
\[
\bX_n \stackrel{{\cal D}}{\longrightarrow} \Po(\blambda \bA),
\]
where in case (b) the matrix $\bA$ is given in Proposition \ref{prop:doob}.
\end{theorem}

\section{Proofs}\label{section:proofs}

Before the proofs, we gather some simple inequalities, which we use frequently
without further reference.
If $a_k,b_k\in [-1,1]$, $k=1,\dots,n$, then
\[
\left| \prod_{k=1}^n a_k - \prod_{k=1}^n b_k \right|
\leq \sum_{k=1}^n | a_k - b_k |.
\]
For $x \in (-1,1)$ we have $|\ee^{x} - 1 - x | \leq x^2$.

For a vector-vector function $\bH: \RR^d \to \RR^d$ the symbol $\nabla \bH$ denotes
\[
\begin{bmatrix}
\frac{\partial}{\partial x_1} H_1 & \ldots & \frac{\partial}{\partial x_d} H_1 \\
\vdots & & \vdots \\
\frac{\partial}{\partial x_1} H_d & \ldots & \frac{\partial}{\partial x_d} H_d \\
\end{bmatrix}.
\]
By the multivariate mean-value theorem, and the monotonicity of the derivatives,
for a vector of generating functions $\bG =(\bG_1, \ldots, \bG_d)$,
for $\bx \in [0,1]^d$
\begin{equation} \label{eq:ineq-nabla}
\bbone - \bG(\bx) \leq  ( \bbone - \bx) \nabla \bG (\bbone)^\top.  
\end{equation}

\subsection{Proofs for Section \ref{section:general}}

Since $\| \bB_{[j,n]} \| \to 0$ for any $j$, we may and do assume that
$n_0 =1$ in conditions \textnormal{(B2)}, \textnormal{(I2)} and \textnormal{(I4)}.

\begin{proof*}{Lemma \ref{lemma:mean-toeplitz}}
Condition (\ref{T1}) simply follows from \textnormal{(B1)},
\textnormal{(B2)} and \textnormal{(B3)}. Conditions (\ref{T2}) and (\ref{T3}) are
the same as \textnormal{(B4)}, and \textnormal{(B5)}, respectively.
\end{proof*}

\begin{remark}
It is worth to note that after rearranging the sum in \textnormal{(B4)}
and using \textnormal{(B3)}
we obtain that \textnormal{(B4)} is equivalent to the convergence
\[
\lim_{n \to \infty} \sum_{j=1}^n \bB_{[j,n]} = \widetilde \bB,
\]
where the relation between $\bA$ and $\widetilde \bB$ is given by
$ (\bB- \bI) \widetilde \bB + \bB = \bA$.
\end{remark}

Recall the definitions (\ref{eq:genfc-F}) and (\ref{eq:acc-Poi}).

\begin{lemma}\label{lemma:F-tildeF}
Assume conditions \textnormal{(B1)}--\textnormal{(B5)},
\textnormal{(M)}. Then for any $\bx \in [0,1]^d$ we have
\[
\lim_{n \to \infty} | F_n (\bx ) - \widetilde F_n(\bx) | = 0.
\]
\end{lemma}

\begin{proof}
Since $\bB_n = \nabla \bG_n (\bbone)$ we have
\[
\nabla \overline \bG_{j+1,n}(\bbone) =
\nabla \bG_{j+1}(\bbone) \nabla \bG_{j+2}(\bbone) \ldots 
\nabla \bG_n(\bbone) = \bB_{[j,n]}.
\]
By the scalar version of (\ref{eq:ineq-nabla}) we have 
\begin{eqnarray*}
| F_n (\bx) - \widetilde F_n( \bx) |
& \leq & \sum_{j=1}^n \left| \ee^{H_j ( \overline \bG_{j+1,n}(\bx)) - 1} -
1 - (H_j (\overline \bG_{j+1,n}(\bx)) - 1 ) \right| \\
& \leq & \sum_{j=1}^n
\left( H_j ( \overline \bG_{j+1,n} (\bx) ) -1 \right)^2 \\
& \leq & \sum_{j=1}^n
\left(  (\bbone - \overline \bG_{[j+1,n]}(\bx)) \bmm_j^\top  \right)^2 \\
& \leq & \sum_{j=1}^n
\left( (\bbone - \bx) \bB_{[j,n]}^\top \bmm_j^\top  \right)^2.
\end{eqnarray*}
Since
\[
\left| (\bbone - \bx) \bB_{[j,n]}^\top \bmm_j^\top  \right|
\leq \| \bbone \| \cdot \| \bA_{j,n} \| \max_{k \geq 1} \| \bmm_k (\bB - \bB_k)^{-1} \|,  
\]
by (\ref{T1}), (\ref{T3}) and \textnormal{(M)} the sum above converges to 0, as stated.
\end{proof}

\begin{proof*}{Theorem \ref{th:INAR}}
Since the generating function of
$\Po(\blambda) = \Po(\lambda_1) \times \ldots \times \Po(\lambda_d)$
has the form
\[
\ee^{\lambda_1(x_1-1)} \ldots \ee^{\lambda_d(x_d-1)}
= \ee^{(\bx-\bbone) \blambda^\top } , \qquad \bx \in [0,1]^d,
\]
by Lemma \ref{lemma:F-tildeF} we only have to show that
\[
\sum_{j=1}^n  \bmm_j \bB_{[j,n]}   \to  \blambda  \bA,
\]
for all $\bx \in [0,1]^d$. This holds according to Lemma \ref{Toeplitz}
and our assumption.
\end{proof*} 

Since $D^{\bj} H_n(\bbone) = m_{n, \bj}$,
the multivariate Taylor expansion gives the following.

\begin{lemma}\label{Taylor} 
If $\EE(\|\bvare_n \|^k) < \infty$ for some $k \in \NN$ then
for all $\bx \in [0,1]^d$
\[
\begin{split}
H_n(\bx) & = \sum_{\bell \in\ZZ_+^d,\,|\bell| < k}
\frac{m_{n,\bell}}{\bell!} (\bx-\bbone)^{\bell} + R_{n,k}(\bx) \\
& = 1+ \sum_{i=1}^{k-1}\frac{1}{i!}
\sum_{\ell_1, \ldots, \ell_i=1}^d
\frac{\partial^i H_n(\bbone)}
{\partial x_{\ell_1} \ldots \partial x_{\ell_i}}
(x_{\ell_1} - 1) \ldots (x_{\ell_i} - 1) + R_{n,k} (\bx),
\end{split}
\]
where  $\bell ! := \ell_1! \ldots \ell_d!$ for
$\bell = (\ell_1, \ldots, \ell_d) \in\ZZ_+^d$, and
\[
|R_{n,k}(\bx)| \leq \sum_{\bell \in\ZZ_+^d,\,|\bell| = k} 
\frac{m_{n,\bell}}{\bell!} (\bbone - \bx )^{\bell}
\leq d^k \| \bbone  - \bx \|^k \max_{|\bell|=k} D^{\bell} H_n(\bbone).
\]
\end{lemma}

\begin{proof*}{Theorem \ref{th:general-immigration}}
Since the offsprings are Bernoulli distributed
\[
\overline \bG_{j+1,n}(\bx) = \bbone +  (\bx - \bbone) \bB_{[j,n]}^\top,
\]
therefore by Lemma \ref{lemma:F-tildeF} it is enough to show that
the convergence
\[
\sum_{j=1}^n
\left[ H_j ( \overline \bG_{j+1,n}(\bx) ) - 1 \right]
\to 
\sum_{i=1}^{k-1} 
\sum_{\ell_{i+1}, \ldots, \ell_{2i} = 1}^d
\Lambda_{i; \ell_{i+1}, \ldots, \ell_{2i}}
(x_{\ell_{i+1}} - 1) \ldots (x_{\ell_{2i}} - 1)
\]
holds. Using Lemma \ref{Taylor}  we may write
\[
\begin{split}
\sum_{j=1}^n \left[ H_j (\overline \bG_{j+1,n}(\bx)) - 1 \right]
& = \sum_{j=1}^n \sum_{i=1}^{k-1}\frac{1}{i!}
\sum_{\ell_1, \ldots, \ell_i=1}^d
\frac{\partial^i H_j(\bbone)}
{\partial x_{\ell_1} \ldots \partial x_{\ell_i}}
\big( ( \bx - \bbone ) \bB_{[j,n]}^\top \big)_{\ell_1}
\ldots \big( (\bx - \bbone ) \bB_{[j,n]}^\top  \big)_{\ell_i} \\
& \phantom{=}  + \sum_{j=1}^n R_{j,k}\big( \bbone + (\bx - \bbone ) \bB_{[j,n]}^\top \big).
\end{split}
\]
Since, for $m \in \{ 1, \ldots, i\}$
\[
\left( (\bx - \bbone ) \bB_{[j,n]}^\top  \right)_{\ell_m} =
\sum_{\ell_{i+m} = 1}^d \left( \bB_{[j,n]} \right)_{ \ell_m, \ell_{i+m}} (x_{\ell_{i+m} } - 1),
\]
by (\ref{eq:imm-assumption-II}) the first term converges for any $i \in \{1, \ldots, k-1 \}$,
\[
\begin{split}
& \lim_{n \to \infty}
\frac{1}{i!}
\sum_{j=1}^n
\sum_{\ell_1, \ldots, \ell_i=1}^d
\frac{\partial^i H_j(\bbone)}
{\partial x_{\ell_1} \ldots \partial x_{\ell_i}}
( ( \bx - \bbone ) \bB_{[j,n]}^\top )_{\ell_1}
\ldots (  (\bx - \bbone ) \bB_{[j,n]}^\top  )_{\ell_i} \\
& = 
\sum_{\ell_{i+1}, \ldots, \ell_{2i} =1}^d
\Lambda_{i; \ell_{i+1}, \ldots, \ell_{2i}}
(x_{\ell_{i+1}} -1) \ldots (x_{\ell_{2i}} - 1).
\end{split}
\]
Using Lemma \ref{Taylor} for the second term we have
\[
\begin{split}
\sum_{j=1}^n |R_{j,k} (\bbone +  (\bx - \bbone ) \bB_{[j,n]}^\top ) |
& \leq \sum_{j=1}^n d^k \max_{|\bell | = k} D^{\bell} H_j(\bbone) \|\bB_{[j,n]}\|^k \\
& \leq d^k \sum_{j=1}^n  \max_{|\bell | = k} D^{\bell} H_j(\bbone)
\| (\bB - \bB_j)^{-1} \| \cdot   \|\bA_{j,n}\|,
\end{split}
\]
which goes to 0, due to (\ref{eq:imm-assumption-I}).
\end{proof*}

\subsection{Proofs for Section \ref{section:special}}

We start with the case when the limit matrix is the identity.

\begin{lemma}\label{L1norm}
For any $d \geq 1$
there exists positive constant  $C_d$ such that
\[
\sum_{j=1}^n \| \bA_j \| \leq C_d \left\| \sum_{j=1}^n \bA_j \right\|
\]
for all $n \in \NN$ and for all matrices $\bA_j \in \RR_+^{d \times d}$,
$j \in \{1,\dots,n\}$.
\end{lemma}

\begin{proof}
The norms of a finite dimensional vector space are equivalent, hence there are
positive constants $c_d , \tc_d$ such that
\[
c_d \sum_{i=1}^d \sum_{k=1}^d |a_{i,k}|
\leq \| \bA \|
\leq \tc_d \sum_{i=1}^d \sum_{k=1}^d |a_{i,k}|
\]
for all matrices $\bA = (a_{i,j})_{i,j\in\{1,\dots,d\}} \in \RR^{d \times d}$.
Put $(\bA_j)_{i,k}=a_{j;i,k}$. Consequently,
\begin{align*}
\sum_{j=1}^n \| \bA_{j} \|
&\leq \tc_d \sum_{j=1}^n \sum_{i=1}^d \sum_{k=1}^d |a_{j;i,k}|
= \tc_d \sum_{j=1}^n \sum_{i=1}^d \sum_{k=1}^d a_{j;i,k} \\
&= \tc_d \sum_{i=1}^d \sum_{k=1}^d \sum_{j=1}^n a_{j;i,k}
= \tc_d \sum_{i=1}^d \sum_{k=1}^d \left| \sum_{j=1}^n a_{j;i,k} \right|
\leq \frac{\tc_d}{c_d} \left\| \sum_{j=1}^n \bA_j \right\|.
\end{align*}
\end{proof}

\begin{proof*}{Proposition \ref{Xmean}}
Condition (\ref{T1}) follows from \textnormal{(I1)}, \textnormal{(I2)}
and \textnormal{(I3)}, as in the general case.
As we already mentioned (\ref{T2}) is automatic, since
\[
\sum_{j = 1}^n \bA_{j,n}
= \sum_{j = 1}^n \left( \bB_{[j,n]} - \bB_{[j-1, n]} \right) 
= \bI - \bB_{[0,n]}
\to \bI \qquad \text{as $n\to\infty$}
\]
by condition \textnormal{(I3)}.

If
\[
C := \sup_{j \geq 1} \frac{\| \bI - \bB_j \|}{1 - \|\bB_j\|} < \infty
\]
then for all $n \geq j\geq 1$,
\[ 
\| \bA_{j,n} \|
\leq \| \bI - \bB_j \| \cdot \| \bB_{[j,n]} \| 
\leq C  (1 - \|\bB_j\|) \, \| \bB_{[j,n]} \|.
\]
Since $\bB_{[n,n]} = \bI$ we have $\| \bB_{[n,n]} \| = 1$, and
\[
\| \bB_{[j,n]} \| \leq \prod_{k=j+1}^n \| \bB_k \|
\]
for all $n > j$, thus
\begin{align*}
\sum_{j=1}^n \| \bA_{j,n} \|
& \leq C ( 1 - \| \bB_n \| ) + 
C \sum_{j=1}^{n-1} (1- \| \bB_j \|) \prod_{k=j+1}^n \| \bB_k \| \\
&= C \bigg( 1 - \prod_{k=1}^n \| \bB_k \| \bigg)
\leq C,
\end{align*}
and we deduce \eqref{T3}.

Otherwise, if $\bA_{j,n} \in \RR_+^{d \times d}$ for all $n \geq 1$ and
all $j\in \{1, \ldots, n\}$ then by Lemma \ref{L1norm}
\[
\sum\limits_{j=1}^n \| \bA_{j,n} \|
\leq C_d \Biggl\| \sum_{j=1}^n \bA_{j,n} \Biggr\| ,
\]
and \eqref{T2} implies \eqref{T3}.
\end{proof*}

\begin{proof*}{Proposition \ref{Xmean_diag}}
We have to check only \eqref{T3}, since \eqref{T1} and \eqref{T2} follow from
conditions \textnormal{(I1)}--\textnormal{(I3)}.
In this case
\[
\bB_{[j,n]} = \bB_{j+1} \ldots \bB_n  = \bU
\diag \left(\varrho_{[j,n],1}, \ldots, \varrho_{[j,n],d} \right) \bU^\top,
\]
where $\varrho_{[j,n],i} = \varrho_{j+1,i} \ldots \varrho_{n, i}$.
Using again that the norm of a normal element in a $C^*$-algebra equal
to its spectral radius, we have
\begin{align*}
\|\bA_{j,n}\|
&=\|\bU ( \diag( \varrho_{[j,n],1}, \ldots, \varrho_{[j,n],d} )
- \diag( \varrho_{[j-1,n],1}, \ldots, \varrho_{[j-1,n],d} ) ) \bU^\top\| \\
&=\|\diag( (1-\varrho_{j,1}) \varrho_{[j,n],1}, \ldots,
(1-\varrho_{j,d}) \varrho_{[j,n],d} )\|\\
&=\max_{1\leq i \leq d} (1-\varrho_{j,i}) \varrho_{[j,n],i}.
\end{align*}
Thus 
(\ref{T3}) follows from
$\sum_{j=1}^n (1-\varrho_{j,i}) \varrho_{[j,n],i} = 1-\varrho_{1,i}\ldots\varrho_{n,i} \leq 1$,
$i \in \{1,\dots,d\}$.
\end{proof*}

Next we turn to the proofs when $\bB_n = \varrho_n \bB$.
A slight modification of the proof of Theorem 5.2.1 in Doob \cite{doob} gives

\begin{lemma} \label{lemma-Doob}
Assume that $( a_{j,n} )$ satisfies (\ref{eq:1d-toeplitz}),
$\sum_{j=1}^{n-1} | a_{j+1,n} - a_{j,n}| \to 0$, and let $\bB$ be a matrix
such that $\| \bB\| \leq 1$.
Then there exists a matrix $\bA$, such that
\[
\lim_{n \to \infty} \sum_{j=1}^n a_{j,n} \bB^j = \bA.
\]
Moreover $\bA \bB= \bB \bA = \bA = \bA^2$.
\end{lemma}

\begin{proof}
Since for every $j$ the matrix $\|\bB^j\| \leq 1$, the sequence $\sum_{j=1}^n a_{j,n} \bB^j$
is bounded, so there is a subsequence $n_k$ and a limit $\bA$ such that
$$
\sum_{j=1}^{n_k} a_{j,n_k} \bB^j \to \bA \quad \textrm{as } k \to \infty.
$$
Multiplying by $\bB$ we obtain
$$
\sum_{j=1}^{n_k} a_{j,n_k} \bB^{j+1} \to \bB \bA  = \bA \bB\quad \textrm{as } k \to \infty.
$$
Writing $n$ instead of $n_k$, the difference between the two limits is
$$
\sum_{j=1}^{n} a_{j,n} \bB^{j+1} - \sum_{j=1}^{n} a_{j,n} \bB^j
=  a_{n,n}  \bB^{n+1}   -  a_{1,n} \bB +
\sum_{j=2}^n (a_{j-1,n} - a_{j,n} ) \bB^j.
$$
Using that $\bB^j$ is bounded, $a_{1,n} \to 0, a_{n,n} \to 0$ and that
$\sum_{j=1}^{n-1} | a_{j+1,n} - a_{j,n} | \to 0$ we obtain that $\bA \bB= \bB \bA = \bA$.
And so
$$
\left( \sum_{j=1}^n a_{j,n} \bB^j \right) \bA =
\sum_{j=1}^n a_{j,n} \bA
$$
gives that for any other subsequential limit $\bC$, $\bA \bC = \bC \bA = \bA$. Since the roles
are interchangeable, we obtain that there is only one limit matrix, which is idempotent.
\end{proof}

Using the lemma above it is easy to prove Proposition \ref{prop:doob}.

\begin{proof*}{Proposition \ref{prop:doob}}
We only have to check that for $a_{j,n} = \varrho_{[j,n]} ( 1 - \varrho_j)$ the condition
$\sum_{j=1}^{n-1} | a_{j+1,n} - a_{j,n} | \to 0$ satisfied. We have
\[
a_{j+1,n} - a_{j,n} = \varrho_{[j+1,n]} \left[ ( 1 - \varrho_{j+1}) - \varrho_{j+1} ( 1 - \varrho_j) \right],
\]
thus
\[
\sum_{j=1}^{n-1} | a_{j+1,n} - a_{j,n} | =
\sum_{j=1}^{n-1} \frac{ |( 1 - \varrho_{j+1}) - \varrho_{j+1} ( 1 - \varrho_j) |}{ 1 - \varrho_{j+1}}
\varrho_{[j+1,n]} ( 1 - \varrho_{j+1}),
\]
which goes to 0, since
\[
\frac{ |( 1 - \varrho_{n+1}) - \varrho_{n+1} ( 1 - \varrho_n) |}{ 1 - \varrho_{n+1}}
= \left| 1  -   \varrho_{n+1} \frac{ 1 - \varrho_n}{1 - \varrho_{n+1}} \right| \to 0.
\]
\end{proof*}

\begin{proof*}{Theorem \ref{th:spec-2}}
To prove the theorem we only have to show that condition (\ref{T2}) holds
for $\bA_{j,n} =  (\bB - \bB_j) \bB_{[j,n]}$.

By the monotonicity assumptions
\[
\bB_{[j,n]} =  \bB_{j+1} \ldots \bB_n  \leq
\varrho_{j+1} \bB \ldots \varrho_{n} \bB =
\varrho_{[j,n]} \bB^{n-j}
\]
and similarly
\[
\bB_{[j,n]} \geq \vartheta_{[j,n]} \bB^{n-j}.
\]
Keeping in mind that each element
of $\bB - \bB_j$ is non-negative, we have
\[
(1 - \varrho_j) \vartheta_{[j,n]} \bB^{n-j+1} \leq  \bA_{j,n}
\leq  (1 - \vartheta_j) \varrho_{[j,n]} \bB^{n-j+1}.
\]
After summation
\begin{equation} \label{eq:theta-est}
\sum_{j=1}^n \bB^{n-j+1} \vartheta_{[j,n]} ( 1 - \varrho_j) \leq
\sum_{j=1}^n \bA_{j,n} \leq \sum_{j=1}^n \bB^{n-j+1} \varrho_{[j,n]} ( 1 - \vartheta_j).
\end{equation}

First we show that the sequences  $( \vartheta_{[j,n]} ( 1 - \varrho_j) )$ and
$( \varrho_{[j,n]} ( 1 - \vartheta_j) )$ satisfy conditions (\ref{eq:1d-toeplitz}).
According to the assumptions
\begin{equation} \label{eq:theta-rho}
\sum_{j=1}^n \vartheta_{[j,n]} ( 1 - \vartheta_j) \to 1 \quad \textrm{and } \
\sum_{j=1}^n \varrho_{[j,n]} ( 1 - \varrho_j) \to 1.
\end{equation}
Since $\varrho_n \geq \vartheta_n$ we have
\begin{equation} \label{eq:theta}
0 \leq \sum_{j=1}^n \vartheta_{[j,n]} ( \varrho_j - \vartheta_j) =
\sum_{j=1}^n \frac{ \varrho_j - \vartheta_j}{1 - \vartheta_j} (1 - \vartheta_j) \vartheta_{[j,n]}
\to 0,
\end{equation}
as 
$$
\frac{ \varrho_j - \vartheta_j}{1 - \vartheta_j} \leq \frac{ \varrho_j - \vartheta_j}{1 - \varrho_j}
\to 0.
$$
Similarly
\begin{equation} \label{eq:rho}
0 \leq \sum_{j=1}^n \varrho_{[j,n]} ( \varrho_j - \vartheta_j) =
\sum_{j=1}^n \frac{ \varrho_j - \vartheta_j}{1 - \varrho_j} (1 - \varrho_j) \varrho_{[j,n]}
\to 0.
\end{equation}
Noting that $\vartheta_{[j,n]}(1- \varrho_j) = \vartheta_{[j,n]} 
[ ( 1- \vartheta_j) - (\varrho_j - \vartheta_j)]$ and
$\varrho_{[j,n]} (1 - \vartheta_j) = \varrho_{[j,n]} 
[ (1 - \varrho_j) + (\varrho_j -  \vartheta_j)]$, (\ref{eq:theta-rho}) combined with
(\ref{eq:theta}) and with (\ref{eq:rho}) shows that conditions (\ref{eq:1d-toeplitz})
indeed hold.

When the convergence $\bB^n \to \bA$ holds, both the upper and the lower estimation
in (\ref{eq:theta-est}) tends to $\bA$, and the statement follows.

In case (b) the extra condition assures
the convergence of the bounds in (\ref{eq:theta-est}) by Lemma \ref{lemma-Doob},
and the equality of the limits readily follows. 
\end{proof*}

\begin{proof*}{Theorem \ref{th:general-offspring}}
By Lemma \ref{lemma:F-tildeF} we have to check that
\[
\sum_{j=1}^n ( \overline \bG_{j+1,n}( \bx ) - \bbone ) {\bmm}_j^\top  \to 
(\bx - \bbone) \blambda^\top.
\]
By (\ref{eq:ineq-nabla}) we have
\begin{equation} \label{eq:G-ineq}
\bbone - \overline \bG_{j+1,n}(\bx) \leq
( \bbone - \bx) \nabla \overline \bG_{j+1,n}(\bbone)^\top 
= ( \bbone - \bx) \bB_{[j,n]}^\top,
\end{equation}
therefore
\[
\sum_{j=1}^n ( \overline \bG_{j+1,n}(\bx) - \bbone ) \bmm_j^\top \geq
\sum_{j=1}^n (\bx - \bbone) \bB_{[j,n]}^\top \bmm_j^\top
\to (\bx - \bbone) \blambda^\top,
\]
where the last convergence holds under the assumptions of the theorem.

According to (\ref{eq:G-ineq})
$ \overline \bG_{j+1,n}(\bx) \in [ \bbone - \bbone \bB_{[j,n]}^\top, \bbone]$, for
all $\bx \in [0,1]^d$. Again by the mean value theorem and by the monotonicity
of the derivatives
\[
\bbone - \bG_j (\by) \geq 
(\bbone - \by) \, \nabla \bG_j( \bbone - \bbone \bB_{[j,n]}^\top )^\top 
=: (\bbone - \by) \, \bTheta_{j,n}^\top,
\]
for $\by \in  [ \bbone -  \bbone \bB_{[j,n]}^\top, \bbone]$, in particular
\[
\bbone - \bG_j (  \overline \bG_{j+1,n}(\bx) ) \geq
(\bbone -  \overline \bG_{j+1,n}(\bx) ) \, \bTheta_{j,n}^\top,
\]
and so induction gives
\[
\bbone - \overline \bG_{j+1,n}(\bx) \geq
(\bbone - \bx)  \bTheta_{n,n}^\top  \bTheta_{n-1,n}^\top \ldots
\bTheta_{j+1,n}^\top
=: (\bbone - \bx) \bTheta_{[j,n]}^\top ,
\]
so
\[
\sum_{j=1}^n ( \overline \bG_{j+1,n}( \bx) - \bbone )  \bmm_j^\top
\leq \sum_{j=1}^n  (\bx - \bbone) \bTheta_{[j,n]}^\top \bmm_j^\top.
\]
We have to check under what conditions
\[
\sum_{j=1}^n  \bmm_j \bTheta_{[j,n]} \to \blambda.
\]
Clearly $\bTheta_{j,n} \uparrow \bB_j$ as $ n \to \infty$. Introduce
$\bC_{j,n} =  (\bI - \bB_j) \bTheta_{[j,n]}$. Since by definition the elements
of $\bC_{j,n}$ are less then or equal to the elements of $\bA_{j,n}$, so
the only assumption we have to check in order to guarantee the convergence
above is
\[
\sum_{j=1}^n \bC_{j,n} \to \bI.
\]
We have
\[
\begin{split}
\sum_{j=1}^n \bC_{j,n}
& =  ( \bI - \bB_1)  \bTheta_{[1,n]} + ( \bI - \bB_2) \bTheta_{[2,n]}  + \ldots +
( \bI - \bB_{n-1}) \bTheta_{[n-1,n]}  + \bI - \bB_n \\
& = \bI -
\Big[ (\bB_n - \bTheta_{n,n}) + (\bB_{n-1} - \bTheta_{n-1,n}) \bTheta_{[n-1,n]} 
+ \ldots  +  (\bB_2 - \bTheta_{2,n} ) \bTheta_{[2,n]}
+ \bB_1 \bTheta_{[1,n]}  \Big].
\end{split}
\]
We show that the sum in the brackets (note that every term is nonnegative)
converge to the 0 matrix. Let us estimate the $(i,k)$-th element of
$\bB_j - \bTheta_{j,n}$. The mean value theorem and the monotonicity of the
derivatives imply
\begin{eqnarray*}
\left( \bB_j - \bTheta_{j,n} \right)_{i,k}
& = & \frac{\partial}{\partial x_k} G_{j,i}( \bbone) -
\frac{\partial}{\partial x_k} G_{j,i}( \bbone -  \bbone \bB_{[j,n]}^\top) \\
& \leq &  
( \bbone \bB_{[j,n]}^\top )
\bigg( \frac{\partial^2}{\partial x_k \partial x_1} G_{j,i}(\bbone), \ldots,
\frac{\partial^2}{\partial x_k \partial x_d} G_{j,i}(\bbone) \bigg)^\top,
\end{eqnarray*}
thus
\[
\left( \bB_j- \bTheta_{j,n} \right)_{i,k} \leq d \, \| \bbone \| \, m_2(j).
\]
So finally we obtain
\[
\sum_{j=1}^n (\bB_j - \bTheta_{j,n} ) \, \bTheta_{[j,n]}  \leq
d \, \| \bbone \| \sum_{j=1}^n m_2(j)
\begin{bmatrix}
1 & \ldots & 1 \\
\vdots & \ddots  & \vdots \\
1 & \ldots & 1
\end{bmatrix}
\bB_{[j,n]},
\]
which goes to the 0 matrix, whenever
$\|(\bI - \bB_n)^{-1}\| \, m_2(n) \to 0$.
\end{proof*}

\section*{Acknowledgement}

The work of the first author  was partially supported by the European Union and the
European Social Fund through project FuturICT.hu (grant no.: T\'AMOP-4.2.2.C-11/1/KONV-2012-0013).
The second, third and fourth authors have been partially supported by the Hungarian
Chinese Intergovernmental S \& T Cooperation Programme 2011--2013.
The second author has been supported by the 
T\'AMOP-4.2.2.C-11/1/KONV-2012-0001 project. The project has been 
supported by the European Union, co-financed by the European Social Fund.
The third and fourth author have been supported by
the European Union and co-funded by the European Social Fund
under the project `Telemedicine-focused research activities on the field of Mathematics,
Informatics and Medical sciences' of project number
`T\'AMOP-4.2.2.A-11/1/KONV-2012-0073'.
The third author's research
was realized in the frames of T\'AMOP 4.2.4. A/2-11-1-2012-0001 `National Excellence
Program – Elaborating and operating an inland student and researcher personal support
system'. The project was subsidized by the European Union and co-financed by the
European Social Fund.
The third author was partially supported by the Hungarian Scientific Research Fund OTKA PD106181.
The fourth author has been supported by the Hungarian Scientific Research Fund
under Grant No.~OTKA T-079128.

\end{document}